\documentclass[12pt,a4paper]{article}
\usepackage{amsmath,amssymb,centernot, amsthm}
\usepackage{latexsym}

\newtheorem{thm}{Theorem}[section]
\newtheorem*{main1}{Main Theorem 1}
\newtheorem*{main2}{Main Theorem 2}
\newtheorem{lem}[thm]{Lemma}
\newtheorem{prop}[thm]{Proposition}
\newtheorem{cor}[thm]{Corollary}

\newtheorem{defi}[thm]{Definition}

\newtheorem{re}[thm]{Result}
\newtheorem{question}[thm]{Question}

\newtheorem{example}[thm]{Example}

\newtheorem{remark}[thm]{Remark}
\newenvironment{rmk}{\begin{remark} \em}{\end{remark}}

\newcommand{\C}{\mathbb{C}}
\newcommand{\Q}{\mathbb{Q}}

\DeclareMathOperator{\ord}{ {\rm ord} }

\begin{document} 
\title{Upper Bounds for Cyclotomic Numbers} 
\author{Tai Do Duc\\ Division of Mathematical Sciences\\
School of Physical \& Mathematical Sciences\\
Nanyang Technological University\\
Singapore 637371\\
Republic of Singapore\\[5mm]
Ka Hin Leung \footnote{Research is supported by grant R-146-000-276-114, Ministry of Education, Singapore}\\ Department of Mathematics\\ 
National University of Singapore\\
Kent Ridge, Singapore 119260\\
Republic of Singapore\\[5mm]
Bernhard Schmidt \footnote{Research is supported by grant RG27/18 (S), Ministry of Education, Singapore}
\\ Division of Mathematical Sciences\\
School of Physical \& Mathematical Sciences\\
Nanyang Technological University\\
Singapore 637371\\
Republic of Singapore}

\maketitle
\newpage

\begin{abstract}
Let $q$ be a power of a prime $p$, let $k$ be a nontrivial divisor of $q-1$ and write $e=(q-1)/k$. 
We study upper bounds for cyclotomic numbers $(a,b)$ of order $e$
over the finite field $\mathbb{F}_q$. 
A general result of our study is that $(a,b)\leq 3$ for all $a,b \in \mathbb{Z}$ 
if $p> (\sqrt{14})^{k/\ord_k(p)}$. More conclusive results will be obtained  through
seperate investigation of the five types of cyclotomic numbers:
 $(0,0), (0,a), (a,0), (a,a)$ and $(a,b)$, where $a\neq b$ and $a,b \in \{1,\dots,e-1\}$.
The main idea we use  is to transform equations over $\mathbb{F}_q$ into equations
 over the field of complex numbers on which we have more information. 
 A major tool for the improvements 
 we obtain over known results is new upper bounds on the norm of cyclotomic integers.  
\end{abstract}

\noindent $2010$ Mathematics Subject Classification 11T22 (primary), 11C20 (secondary)  


\section{Introduction and Definitions}

First, we fix some notations and definitions. By $q$ we denote a power of a prime $p$. Let $e$ and $k$ be nontrivial divisors of $q-1$ such that $q=ek+1$. Let $g$ denote a primitive element of the finite field $\mathbb{F}_q$. For each $a \in \mathbb{Z}$, write 
\begin{equation} \label{class of powers congruent to a}
C_a=\{g^a,g^{a+e},...,g^{a+(k-1)e}\}.
\end{equation}
As $C_a=C_{a+e}$, we only need to consider the sets $C_a$ with $a\in \{0,1,\dots,e-1\}$. 

\begin{defi} \label{cyc}
For $a,b \in \{0,1,...,e-1\}$, define $(a,b)$ as the number of solutions to the equation 
$$1+x=y, \ x\in C_a, \ y\in C_b.$$ 
Equivalently, this is the number of pairs $(r,s)$ with $0 \leq r,s \leq k-1$ such that
\begin{equation} \label{cyclotomic number (a,b)}
1+g^{a+re}=g^{b+se}.
\end{equation}
The number $(a,b)$ is called a \textbf{cyclotomic number} of order $e$.
\end{defi}

Cyclotomic numbers have been studied for decades by many authors, as they have applications in various areas. 
These numbers can be used to compute Jacobi sums, and vice versa, see \cite{agr}. Vandiver \cite{leh}, \cite{van1}, \cite{van2}, \cite{van3}, \cite{van4} related cyclotomic numbers to Fermat's Last Theorem and proved the theorem for exponents $\leq 2000$. 
Cyclotomic classes $C_a$ were used by Paley \cite{pal} in 1993 to construct difference sets. This approach was later employed by many other authors. Storer's book \cite{sto} summarizes the results in this direction up to $1967$. 
In the 1960s to 1980s, Baumert, Whiteman, Evans et al. explicitly determined all numbers $(a,b)$ of orders $e\leq 12$ and $e=14,15,16,18,20,24$.

\medskip

Under asymptotic conditions, cyclotomic numbers exhibit an interesting uniform behaviour. 
Katre \cite{kat} proved in $1989$ that, for  fixed $e$ and $q\to \infty$, we asymptotically have  $(a,b)\approx q/e^2$ for all $a,b\in \mathbb{Z}$. 
On the other hand, fixing $k$, it was proved by Betshumiya et al. \cite{bet} that $(0,0) \leq 2$ if $p$ is \textit{sufficiently large} compared to $k$.
 In this paper, the condition ``sufficiently large'' is not explicitly specified and, in fact, the lower bound on $p$
 required for their method is difficult to write down explicitly. 
 The goal of our paper is to find simple and improved lower bound on $p$  which guarantees that all the numbers $(a,b)$ are small. 
 The following theorem is a main result of our study. 

\begin{main1} \label{general}
Let $q$ be a power of a prime $p$. Let $e$ and $k$ be nontrivial divisors of $q-1$ such that $q=ek+1$. If $$p>\left(\sqrt{14}\right)^{k/\ord_k(p)},$$ 
then 
$(a,b)\leq 3 \ \ \text{for all} \ \ a,b\in \mathbb{Z}$.
\end{main1}

If $k$ is a prime, we obtain a better bound as follows.

\begin{main2} \label{general2}
Let $q$ be a power of a prime $p$. Let $e$ and $k$ be nontrivial divisors of $q-1$ such that $k$ is a prime and $q=ek+1$. If 
$$p>(3^{k-1}k)^{1/\ord_k(p)},$$
then 
$$(a,b) \leq 2 \ \ \text{for all} \ \ a,b\in \mathbb{Z}.$$
\end{main2}

\medskip

We continue with introducing some notation and results we need later. 
For a positive integer $k$, let $\zeta_k$ denote a complex primitive $k$th root of unity. A square matrix is called \textbf{circulant} if each of its rows (except the first) is obtained from the previous row by shifting the entries one position to the right and moving the last entry to the front. Moreover, given a matrix $H$, we denote the conjugate transpose of $H$ by $H^*$. The following result about eigenvalues and eigenvectors of a circulant matrix is well known, see \cite{dav}, for example.

\begin{re} \label{cir}
Let $k$ be a positive integer and let $M$ be a circulant matrix with the first row $(a_0,\dots,a_{k-1})$ 
where $a_0,\ldots,a_{k-1}\in \C$.  Then the eigenvalues and eigenvectors of $M$ are
$$\lambda_i=\sum_{j=0}^{k-1}a_j\zeta_k^{ij}, \ X_i=(1,\zeta_k^i,\dots,\zeta_k^{i(k-1)})^T \  \text{for} \ 0\leq i \leq k-1.$$
\end{re}

\medskip

In the next section, we review some results on vanishing sums of roots of unity which will be needed for our study. The following terminology was used in \cite{con}. 
Let $T$ be a finite set of complex roots of unity and let $c_{\alpha}$, $\alpha \in T$, be nonzero
rational numbers. 
The sum 
$$S=\sum_{\alpha\in T} c_\alpha \alpha, \ c_\alpha\in \mathbb{Q}\setminus \{0\},$$
is called a \textbf{vanishing sum} of roots of unity if $S=0$. 
We say that $S$ is \textbf{nonempty} if $T\neq \emptyset$.
 The \textbf{length} $l(S)$ is the cardinality of $T$.
 The \textbf{exponent} $e(S)$ denotes the least common multiple of all orders of the roots of unity $\alpha\in T$. 
 We say that $S$ is \textbf{similar} to any sum of the form $k\cdot \beta S'$, where $k \in \mathbb{Q}\setminus\{0\}$ and $\beta$ is a root of unity and $S'$ has the form 
 $$S'=\sum_{\alpha\in T} (\varepsilon_\alpha c_\alpha) (\varepsilon_\alpha \alpha), \ \text{where} \ \ \varepsilon_\alpha \in \{1,-1\}.$$
 We call the vanishing sum $S$ \textbf{minimal} if $S$ contains no vanishing subsum. 
 The sum $S$ is a \textbf{reduced sum} if  $\alpha=1$ for some $\alpha\in T$.

 \bigskip

\section{Vanishing Sums of Roots of Unity}
The following result states that a minimal vanishing sum of roots of unity is similar to a vanishing sum whose order is squarefree, see \cite[Corollary 3.2]{lam} or \cite[Theorem 1]{con} for a proof.
 
\begin{re} \label{squarefree}
If $S=\alpha_1+\cdots+\alpha_n$ is a minimal vanishing sum of $m$th roots of unity, 
then after multiplying $S$ by a suitable $m$th root of unity, we may assume 
that all $\alpha_i$ are $m_0$th roots of unity, where $m_0$ is the largest square-free divisor of $m$. 
\end{re}

\medskip

The next result is part of \cite[Theorem 6]{con} and will be useful for our study. 
 
\begin{re} \label{vanishing sum of small length}
Let $S$ be a nonempty vanishing sum of length at most $6$ that does
not contain subsums similar to $1 + (-1)$ or $1+\zeta_3+\zeta_3^2$.
Then  $S$ is similar to one of the sums
$$1+\zeta_5+\zeta_5^2+\zeta_5^3+\zeta_5^4, $$  
$$-\zeta_3-\zeta_3^2+\zeta_5+\zeta_5^2+\zeta_5^3+\zeta_5^4.$$
\end{re}

\bigskip

\section{Bounds on Norms of Cyclotomic Integers}
A \textbf{cyclotomic integer} (not to be confused with a cyclotomic number) is
an algebraic integer in a cyclotomic field. Every cyclotomic integer 
can be written as a sum of complex roots of unity.
The improvements over the previously known results  we obtain
arise from new bounds on absolute norms of cyclotomic integers. 
First, we discuss a general norm bound.

\medskip

Note that every cyclotomic integer in  $\Q(\zeta_k)$ can be written as $f(\zeta_k)$, where
$f(x)=\sum_{i=0}^{k-1}a_ix^i$ is a polynomial with integer coefficients.
Since $|f(\zeta_k^j)| \le \sum_{i=0}^{k-1} |a_i|$, an obvious bound for the absolute norm
of $f(\zeta_k)$ is 
\begin{equation} \label{obvious}
|N(f(\zeta_k))| = \left| \prod_{j: \gcd(j,k)=1} f(\zeta_k^j)\right| \le \left(\sum_{i=0}^{k-1} |a_i|\right)^{\varphi(k)}.
\end{equation}
In this section, we provide some stronger bounds that are suitable for the applications
to cyclotomic numbers we are interested in.

\begin{thm} \label{bound for norm of an algebraic integer}
Let $k$ be a positive integer, let $f(x)=\sum_{i=0}^{k-1}a_ix^i \in \mathbb{Z}[x]$ and let $N$ denote the absolute norm of $\mathbb{Q}(\zeta_k)$. Then
\begin{equation} \label{bound for norm}
|N(f(\zeta_k))| \leq \left( \frac{k}{\varphi(k)}\sum_{i=0}^{k-1}a_i^2 \right)^{\varphi(k)/2}.
\end{equation}
In particular, if $\sum_{i=0}^{k-1}a_i^2 \geq 3$, then
\begin{equation} \label{bound for norm - resulting inequality}
|N(f(\zeta_k))| \leq \left( \sum_{i=0}^{k-1} a_i^2 \right)^{k/2}.
\end{equation}
\end{thm}

\begin{proof}
We have
$$\sum_{h=0}^{k-1} |f(\zeta_k^h)|^2=\sum_{i,j,h=0}^{k-1}a_ia_j\zeta_k^{(i-j)h}=k\sum_{i=0}^{k-1}a_i^2.$$
By the inequality between arithmetic and geometric means, we have
$$|N(f(\zeta_k))|=|\prod_{(h,k)=1} f(\zeta_k^h)| \leq \left(\frac{\sum_{(h,k)=1}|f(\zeta_k^h)|^2}{\varphi(k)}\right)^{\varphi(k)/2} \leq \left(\frac{k}{\varphi(k)}\sum_{i=0}^{k-1}a_i^2 \right)^{\varphi(k)/2},$$
which proves (\ref{bound for norm}). \\
Now consider the case $S=\sum_{i=0}^{k-1}a_i^2 \geq 3$. Since $g(x)=(kS/x)^{x/2}$ is increasing over the interval $[1,k]$, we obtain
$$|N(f(\zeta_k))|\leq g(\varphi(k)) \leq g(k)=S^{k/2}.$$
\end{proof}

\medskip

In the case $k$ is a prime, we obtain a different bound on the norm of $f(\zeta_k)$ in the next theorem. 
This bound is better than (\ref{bound for norm}) in certain situations. 

\medskip

For the rest of this section, we assume that $k$ is a prime. For $f(x)=\sum_{i=0}^{k-1}a_ix^i$, let $M$ denote the circulant matrix whose first row is $(a_0,\dots,a_{k-1})$ and let $N$ denote the $(k-1)\times (k-1)$ matrix obtained from $M$ by deleting its first row and its first column. To find an upper bound for $|N(f(\zeta_k))|$, we first find a relation between $N(f(\zeta_k))$ and $\det(M)$ or $\det(N)$. Then an upper bound for $|\det(M)|$ or $|\det(N)|$ will give us an upper bound for $|N(f(\zeta_k))|$. 

\medskip

 Bounds for the determinant of a matrix are abundant in the literature. We only need the following result by Schinzel \cite{schinzel}.

\begin{re} \label{schinzel bound}
Let $N=(a_{ij})_{i,j=0}^{n-1}$ be an $n\times n$ matrix with real entries. For $i=0,1,...,n-1$, write
$N_i^{+}= \sum_{j=0}^{n-1}\max\{0,a_{ij}\}$ and $N_i^{-} = \sum_{j=0}^{n-1}\max\{0, -a_{ij}\}$.
We have
\begin{equation} \label{determinant bound}
|\det(N)| \leq \prod_{i=0}^{n-1} \max\{N_i^{+}, N_i^{-}\}.
\end{equation}
\end{re}

\begin{prop} \label{bound for norm of a cyclotomic integer 2}
Using the notation introduced above, we have the following
\begin{itemize}
\item[(a)] If  $\sum_{i=0}^{k-1}a_i \neq 0$, then
\begin{equation} \label{bound for norm of alpha 2}
N(f(\zeta_k)) = \frac{\det(M)}{\sum_{i=0}^{k-1}a_i}.
\end{equation}
\item[(b)] If $\sum_{i=0}^{k-1}a_i=0$, then
\begin{equation} \label{bound for norm of alpha 3}
N(f(\zeta_k)) = k\det(N).
\end{equation}
\end{itemize}
\end{prop}

\begin{proof}
For each $0\leq i \leq k-1$, define a column vector
$$X_i=\frac{1}{\sqrt{k}}(1,\zeta_k^i,\zeta_k^{2i},...,\zeta_k^{(k-1)i})^T.$$
By Result \ref{cir}, the eigenvalues of $M$ are $\lambda_i=f(\zeta_k^i)$ and the corresponding eigenvectors are $X_i$, $0\leq i \leq k-1$. Since $k$ is a prime, we have
\begin{equation} \label{norm prime}
N(f(\zeta_k))=\prod_{i=1}^{k-1}f(\zeta_k^i)=\prod_{i=1}^{k-1} \lambda_i.
\end{equation}
Note that $\det(M)=\prod_{i=0}^{k-1}\lambda_i$. If $\lambda_0=\sum_{i=0}^{k-1} a_i \neq 0$, then (\ref{bound for norm of alpha 2}) is clear.

\medskip

 Suppose that $\lambda_0=0$. Note that $X_i^*X_j=1$ if $i=j$ and $X_i^*X_j=0$ if $i \neq j$. 
 Let $Q$ be the $k\times k$ matrix with columns $X_0,...,X_{k-1}$, 
 then $Q^{-1}$ is the $k\times k$ matrix with rows $X_0^*,\dots,X_{k-1}^*$. We have 
\begin{equation*} \label{diagonalization of M}
M=Q\begin{pmatrix}
\lambda_0\\
&\lambda_1\\
&&\ddots\\
&&&\lambda_{k-1}
\end{pmatrix}Q^{-1}.
\end{equation*}
By the definition of $N$, 
\begin{equation*} \label{form of N}
N=Q_1\begin{pmatrix}
\lambda_0\\
&\lambda_1\\
&&\ddots\\
&&&\lambda_{k-1}
\end{pmatrix}Q_1^{'},
\end{equation*}
where $Q_1$ is the $(k-1)\times k$ matrix formed by the last $k-1$ rows of $Q$ and $Q_1^{'}$ is the $k\times (k-1)$ matrix formed by the last $k-1$ columns of $Q^{-1}$. Since $\lambda_0=0$, we have
$$N=Q_2 \begin{pmatrix}
\lambda_1\\
&&\ddots\\
&&&\lambda_{k-1}
\end{pmatrix}Q_2^{'},$$
where $Q_2$ is the $(k-1)\times (k-1)$ matrix formed by the last $k-1$ columns of $Q_1$ and $Q_2{'}$ is the matrix formed by the last $k-1$ rows of $Q'_1$. We obtain
$$\det(N)=\det(Q_2Q_2^{'})\prod_{i=1}^{k-1} \lambda_i.$$
By (\ref{norm prime}), the equation (\ref{bound for norm of alpha 3}) is equivalent to $\det(Q_2Q_2^{'})=1/k$. Note that  $(Q_2^{'})_{ij}=\overline{(Q_2)}_{ij}$ for any $i,j$, as $Q_2$ and $Q_2^{'}$ are submatrices of $Q$ and $Q^{-1}$, respectively. More precisely, we have
$$Q_2=\frac{1}{\sqrt{k}}\begin{pmatrix}
\zeta_k & \zeta_k^2 & \cdots & \zeta_k^{k-1} \\
\zeta_k^2 & \zeta_k^4 & \cdots & \zeta_k^{2(k-1)} \\
&&\ddots \\
\zeta_k^{k-1} & \zeta_k^{2(k-1)} & \cdots & \zeta_k^{(k-1)(k-1)} 
\end{pmatrix}.$$
The $(i,j)$th entry of $Q_2Q_2^{'}$ is 
$$\frac{1}{k}\sum_{t=1}^{k-1}\zeta_k^{(i-j)t}=\begin{cases} (k-1)/k \ \text{if} \ i=j\\
-1/k \ \text{if} \ i \neq j
\end{cases}.
$$
Hence $Q_2Q_2^{'}$ is a circulant matrix of size $(k-1)\times (k-1)$ with the first row is $((k-1)/k, -1/k, ...,-1/k)$. By Result \ref{cir}, the eigenvalues of $Q_2Q_2^{'}$ are 
$$\beta_j=\frac{1}{k}(k-1-\sum_{i=1}^{k-2}\zeta_{k-1}^{ij})=\begin{cases} 1/k \ \ \text{if} \ \ j=0, \\ 1 \ \ \text{if} \ \ 1 \leq j \leq k-2. \end{cases}$$
We obtain
 $$\det(Q_2Q_2^{'})=\prod_{j=0}^{k-2}\beta_j=1/k.$$
\end{proof}

Combining Result \ref{schinzel bound} and Proposition \ref{bound for norm of a cyclotomic integer 2},
we get the following norm bound, which in numerous cases is stronger 
than Theorem \ref{bound for norm of an algebraic integer}.

\begin{cor} \label{secondnormbound}
Let $k$ be a prime and let $f(x)=\sum_{i=0}^{k-1}a_ix^i \in \mathbb{Z}[x]$.
Write
$A^{+}= \sum_{j=0}^{n-1}\max\{0,a_j\}$, $A^{-} = \sum_{j=0}^{n-1}\max\{0, a_{j}\}$,
and $A=\max\{A^+,A^-\}$. 

\medskip 

\noindent
(a) 
If $\sum_{i=0}^{k-1}a_i=0$, then 
$$|N(f(\zeta_k))| \leq kA^{k-1}.$$
(b)
If $\sum_{i=0}^{k-1}a_i\neq 0$, then 
$$|N(f(\zeta_k))| \leq \frac{A^{k}} {\sum_{i=0}^{k-1}a_i}.$$
\end{cor}


\section{Equations over {\boldmath $\mathbb{F}_q$} and {\boldmath $\mathbb{C}$} }
The following theorem shows that under some condition on the characteristic of the 
finite field  $\mathbb{F}_q$,
we can transform certain equations over $\mathbb{F}_q$ 
to equations over the field of complex numbers $\mathbb{C}$, and vice versa.

\begin{thm} \label{equivalence of equations over two fields}
Let $q$ be a power of a prime $p$ and let $e,k$ be 
nontrivial divisors of $q-1$ such that $q=ek+1$. 
Let $g$ be a primitive element of $\mathbb{F}_q$ and let $f(x)=\sum_{i=0}^{k-1}a_ix^i \in \mathbb{Z}[x]$. Suppose that
\begin{equation} \label{necessary conditions for equivalences of equations}
p > \left( \frac{k}{\varphi(k)}\sum_{i=0}^{k-1}a_i^2 \right)^{\frac{\varphi(k)}{2{\ord_k(p)}}},
\end{equation}
then $f(g^e)=0$ over $\mathbb{F}_q$ if and only if $f(\zeta_k)=0$ over $\mathbb{C}$.

\medskip

In particular, the same conclusion holds if $\sum_{i=0}^{k-1}a_i^2 \geq 3$ and 
\begin{equation} \label{necessary resulting}
p>\left(\sum_{i=0}^{k-1} a_i^2 \right)^{\frac{k}{2\ord_k(p)}}.
\end{equation}
\end{thm}

\begin{proof}
Let $\mathfrak{p}$ be a prime ideal of $\mathbb{Z}[\zeta_k]$ that contains $p$. Write $q=p^n$ and $b=\ord_k(p)$. Note that $b$ divides $n$ because $q=p^n \equiv 1 \pmod{k}$.  Since $\mathbb{Z}[\zeta_k]/{\mathfrak{p}}$ is a finite field extension of $\mathbb{Z}/p\mathbb{Z}$ of order $b$, we have $\mathbb{Z}[\zeta_k]/{\mathfrak{p}} \cong \mathbb{F}_{p^b}$. Let $\phi: \mathbb{F}_{p^b} \rightarrow \mathbb{Z}[\zeta_k]/{\mathfrak{p}}$ be an isomorphism. Note that $g^e$ is a primitive $k$th root of unity in $F_{p^b}$, so $\phi(g^e)$ is also a primitive $k$th root of unity in $\mathbb{Z}[\zeta_k]/ \mathfrak{p}$, which implies $\phi(g^e)=\zeta_k^j+\mathfrak{p}$ for some integer $j$ coprime to $k$. We have
\begin{equation} \label{equivalence equation}
f(g^e)=0 \ \text{over} \ \mathbb{F}_q \Leftrightarrow \phi(f(g^e))=f(\zeta_k^j)+\mathfrak{p}=0 \ \text{in} \ \mathbb{Z}[\zeta_k]/\mathfrak{p} \Leftrightarrow f(\zeta_k^j) \in \mathfrak{p}.
\end{equation}
Suppose that $f(\zeta_k)=0$ over $\mathbb{C}$. We have $f(\zeta_k^j)=0$, as $j$ is coprime to $k$. By (\ref{equivalence equation}), $f(g^e)=0$ over $\mathbb{F}_q$. Now assume that $f(g^e)=0$ over $\mathbb{F}_q$. Note that $N(\mathfrak{p})=p^b$, where by $N(\mathfrak{p})$ we mean the norm of the ideal $\mathfrak{p}$ in $\mathbb{Z}[\zeta_k]$. By (\ref{equivalence equation}), we have $N(f(\zeta_k^j)) \equiv 0 \pmod{p^b}$. As $j$ is coprime to $k$, we have $N(f(\zeta_k^j))=N(f(\zeta_k))$. Thus
\begin{equation} \label{congruence for norm f(zetak)}
N(f(\zeta_k)) \equiv 0 \pmod{p^b}.
\end{equation}
On the other hand, by Theorem \ref{bound for norm of an algebraic integer} we have
\begin{equation} \label{theorem 1}
|N(f(\zeta_k))| \leq \left( \frac{k}{\varphi(k)}\sum_{i=0}^{k-1}a_i^2\right)^{\varphi(k)/2}.
\end{equation}
If $f(\zeta_k) \neq 0$, then $N(f(\zeta_k)) \neq 0$ and (\ref{congruence for norm f(zetak)}), (\ref{theorem 1}) imply
$$p^b \leq \left( \frac{k}{\varphi(k)}\sum_{i=0}^{k-1}a_i^2\right)^{\varphi(k)/2},$$
contradicting (\ref{necessary conditions for equivalences of equations}). Therefore, $f(\zeta_k)=0$. \\

Lastly, the conclusion for the case $\sum_{i=0}^{k-1}a_i^2\geq 3$ follows from (\ref{bound for norm - resulting inequality}).
\end{proof}

\medskip

The next theorem follows from Corollary \ref{secondnormbound} in the same
way as Theorem \ref{equivalence of equations over two fields} follows from Theorem \ref{bound for norm of an algebraic integer},
so we skip the proof.

\begin{thm} \label{equivalence of equations over two fields 2}
Let $q$ be a power of a prime $p$ and let $e,k$ be 
nontrivial divisors of $q-1$ such that $q=ek+1$ and $k$ is a prime.
Let $g$ be a primitive element of $\mathbb{F}_q$ and let $f(x)=\sum_{i=0}^{k-1}a_ix^i \in \mathbb{Z}[x]$.
Write
$A^{+}= \sum_{j=0}^{n-1}\max\{0,a_j\}$, $A^{-} = \sum_{j=0}^{n-1}\max\{0, a_{j}\}$,
and $A=\max\{A^+,A^-\}$. 
Suppose that one of the following conditions holds.

\noindent \item[(a)] $\sum_{i=0}^{k-1}a_i=0$ and
\begin{equation} \label{necessary 3rd condition for equivalences of equations}
p^{\ord_k(p)}> kA^{k-1}.
\end{equation}

\noindent \item[(b)] $\sum_{i=0}^{k-1} a_i \neq 0$ and 
\begin{equation} \label{necessary 2nd condition for equivalences of equations}
p^{\ord_k(p)}> \frac{A^k}{|\sum_{i=0}^{k-1}a_i|}.
\end{equation}
Then we have $f(g^e)=0$ over $\mathbb{F}_q$ if and only if $f(\zeta_k)=0$ over $\mathbb{C}$.
\end{thm}

\bigskip

\section{Upper Bounds for Cyclotomic Numbers}
In this section, we apply Theorem \ref{equivalence of equations over two fields} to derive upper 
bounds for cyclotomic numbers $(a,b)$. 
In Theorem \ref{bound for norm of an algebraic integer}, 
the upper bound $(k/\varphi(k) \sum a_i^2)^{\varphi(k)/2}$ is largest when $\varphi(k)$ is approximately $k$.
 Thus, in this case, an improved bound is desirable and, in particular, when $k$ is a prime.
Theorem \ref{equivalence of equations over two fields 2} 
will come into play in this situation and we will discuss this case separately in the last section. 

\medskip

Note that $(a,b)=(a',b')$ whenever $a\equiv a' \pmod{e}$ and $b\equiv b' \pmod{e}$. From now on, we always assume that $a,b \in \{0,1,\dots,e-1\}$. First, we recall the main result of this section.

\begin{main1} \label{general}
Let $q$ be a power of a prime $p$. Let $e$ and $k$ be nontrivial divisors of $q-1$ such that $q=ek+1$. If
\begin{equation} \label{bound pk}
p>\left(\sqrt{14}\right)^{k/\ord_k(p)},
\end{equation}
then 
\begin{equation} \label{value (a,b)}
(a,b)\leq 3 \ \ \text{for all} \ \ a,b\in \mathbb{Z}.
\end{equation}
\end{main1}

Our proof for this theorem is divided into five cases: We separately
investigate cyclotomic numbers 
 $(0,0), (0,a), (a,0), (a,a)$ and $(a,b)$ where $a\neq b$ and $a,b \in \{1,\dots,e-1\}$. 
In fact,  in each case, we obtain a stronger result than Main Theorem 1, which is just a simplified
consequence of the analysis of the different cases.

\medskip

\begin{thm} \label{the number (0,0)}
If 
\begin{equation} \label{necessary conditions for (0,0)}
p> \left( \frac{3k}{\varphi(k)} \right)^{\frac{\varphi(k)}{2\ord_k(p)}},
\end{equation}
then 
\begin{equation} \label{result for (0,0)}
(0,0)=\begin{cases} 0 \ \text{if} \ k \not\equiv 0 \pmod{6} \ \text{and} \ 2 \not\in C_0, \\
1 \ \text{if} \ k \not\equiv 0 \pmod{6} \ \text{and} \ 2 \in C_0, \\
2 \ \text{if} \ k \equiv 0 \pmod{6} \ \text{and} \ 2 \not\in C_0, \\
3 \ \text{if} \ k \equiv 0 \pmod{6} \ \text{and} \ 2 \in C_0.
\end{cases}
\end{equation}
\end{thm}
\begin{proof}
Suppose that there are $0 \leq a,b \leq k-1$ with $1+g^{ae}=g^{be}$. 
Then $2 \in C_0$ if $a=0$. Thus in the case $2 \in C_0$, there is one solution to $1+g^{ae}=g^{be}$ in which $a=0$.

From now on, suppose that $a \neq 0$ and $1+g^{ae}=g^{be}$. We have $b \not\in \{0,a\}$ and $f(x)=1+x^a-x^b$ is a polynomial of degree at most $k-1$ with two coefficients $1$, one coefficient $-1$ and all other coefficients $0$. Write $f(x)=\sum_{i=0}^{k-1} a_ix^i$, then $\sum_{i=0}^{k-1}a_i^2=3$ and $f(g^e)=0$. By (\ref{necessary conditions for (0,0)}) and Theorem \ref{equivalence of equations over two fields}, we have
$$f(\zeta_k)=1+\zeta_k^a-\zeta_k^b=0.$$
By Result \ref{vanishing sum of small length}, we obtain $1+\zeta_k^a-\zeta_k^b=1+\zeta_3+\zeta_3^2$, which happens only when $6 \mid k$ and $(a,b) \in \{(k/3,k/6),(2k/3,5k/6)\}$, proving (\ref{result for (0,0)}).
\end{proof}

\medskip

Note that by (\ref{necessary resulting}), Theorem \ref{the number (0,0)} still holds
when (\ref{necessary conditions for (0,0)}) is replaced by $p>3^{k/(2{\rm ord}_k(p))}$. 
This shows that Main Theorem 1 holds in the case $(a,b)=(0,0)$.

\medskip

We mentioned in the introduction that Vandiver has used cyclotomic numbers
to obtain results on Fermat's Last Theorem.
The next Corollary gives an example for this kind of argument.
Considering the Diophantine equation $x^e+y^e=z^e$ modulo $p$,
Theorem \ref{the number (0,0)}  implies the following.

\begin{cor} \label{fermat}
If $p$ is a prime with $p=ek+1>3^{k/2}$, 
then $x^e+y^e=z^e$ with $x,y, z \in \mathbb{Z}$ , implies either $2$ is an $e$th power modulo $p$ or $xyz\equiv 0 \pmod{p}$.
\end{cor}

For example, let $p=1301=100\cdot 13+1$ and let $e=100, k=13$. Note that $2$ is not a $100$th power modulo $1301$.
Therefore, if $x^{100}+y^{100}\equiv z^{100} \pmod{1301}$, then $xyz\equiv 0 \pmod{1301}$.

\medskip

\begin{thm} \label{result for (0,a)}
Let $a \in \{1,\dots,e-1\}$. If
\begin{equation} \label{necessary condition for (0,a)}
p>\left( \frac{4k}{\varphi(k)} \right)^{\frac{\varphi(k)}{2\ord_{k}(p)}} ,
\end{equation}
then 
\begin{equation} \label{conclusion for (0,a)}
(0,a) \leq \begin{cases} 3 \ \text{if} \ 2 \in C_a, \\
2 \ \text{if} \ 2 \not\in C_a.
\end{cases}
\end{equation}
\end{thm}

\begin{proof}
Note that $1+g^{ie}=g^{je+a}$ implies $1+g^{-ie}=g^{(j-i)e+a}$, so each solution $(i,j)$ to $1+g^{ie}=g^{je+a}$ induces a solution $(-i,j-i)$ (calculation is modulo $k$) to the same equation, two of which are different if and only if $i \neq 0$. Moreover if $i=0$, then $2=g^{je+a}\in C_a$ and there is one solution to $1+g^{ie}=g^{je+a}$ in which $i=0$.

Suppose that $2 \in C_a$ and $(0,a) \geq 4$. There are two different pairs $(i_1,j_1), (i_2,j_2)$ with 
\begin{equation} \label{condition}
i_1\neq 0, \ i_2 \neq 0, \ (i_2,j_2) \neq (-i_1,j_1-i_1) \ \text{and} \ (i_1,j_1)\neq (-i_2,j_2-i_2)
\end{equation} 
such that $1+g^{i_1e}=g^{j_1e+a}$ and $1+g^{i_2e}=g^{j_2e+a}$. We obtain
\begin{equation} \label{key equation for (0,a)}
1+g^{i_1e}-g^{(j_1-j_2)e}-g^{(j_1-j_2+i_2)e}=0.
\end{equation}
In (\ref{key equation for (0,a)}), the numbers $0, i_1, j_1-j_2$ and $j_1-j_2+i_2$ are pairwise different. Indeed, by (\ref{condition}), we need only to show that $i_1 \neq j_1-j_2+i_2$. If $i_1-i_2=j_1-j_2$, then by subtracting two equations $1+g^{i_1e}=g^{j_1e+a}$ and $1+g^{i_2e}=g^{j_2e+a}$, we obtain $g^{i_2e}=g^{j_2e+a}$, a contradiction as $C_0\cap C_a=\emptyset$.

By (\ref{necessary condition for (0,a)}) and Theorem \ref{equivalence of equations over two fields}, the equation (\ref{key equation for (0,a)}) implies
$$1+\zeta_k^{i_1}-\zeta_k^{j_1-j_2}-\zeta_k^{j_1-j_2+i_2}=0.$$
By Result \ref{vanishing sum of small length}, this is possible only when the sum on left-hand side sum cancels in pairs. This happens only when ``$2 \mid k$ and $i_1=i_2=k/2$" or ``$j_1=j_2$ and $i_1=i_2$", both of which are not possible. Therefore, we obtain $(0,a) \leq 3$ if $2 \in C_a$.

Next, suppose that $2 \not\in C_a$ and $(0,a)\geq 3$. Note that for any $0\leq i,j\leq k-1$ with $1+g^{ie}=g^{je+a}$, we have $i \neq 0$. There exist two pairs $(i_1,j_1)$ and $(i_2,j_2)$ with
$$i_1\neq 0, \ i_2 \neq 0, \ (i_2,j_2) \neq (-i_1,j_1-i_1) \ \text{and} \ (i_1,j_1)\neq (-i_2,j_2-i_2)$$
 such that $1+g^{i_1e}=g^{j_1e+a}$ and $1+g^{i_2e}=g^{j_2e+a}$. We obtain a contradiction by the same argument as in the previous case. 
\end{proof}

\medskip

\begin{thm} \label{result for (a,0)}
Let $a\in \{1,\dots,e-1\}$. If
\begin{equation} \label{necessary for (a,0)}
p>\left( \frac{4k}{\varphi(k)} \right)^\frac{\varphi(k)}{2\ord_k(p)},
\end{equation}
then
\begin{equation} \label{conclusion for (a,0)}
(a,0) \leq \begin{cases} 3 \ \text{if} \ 2 \in C_a \ \text{and} \ 2\mid k,\\
2 \ \text{if} \ 2 \not\in C_a \ \text{and} \ 2 \mid k ,\\
2 \ \text{if} \ 2 \nmid k .
\end{cases}
\end{equation}
\end{thm}

\begin{proof}
First, assume that $k$ is even. If $1+g^{ie+a}=g^{je}$, then $1+g^{(k/2+j)e}=g^{(k/2+i)e+a}$, as $g^{ek/2}=-1$. This implies $(a,0)=(0,a)$ and the conclusion follows from Theorem \ref{result for (0,a)}.

From now on, we assume that $k$ is odd and $(a,0) \geq 3$. For $t=1,2,3$, let $(i_t, j_t)$ be three different pairs with $0\leq i_t,j_t \leq k-1$ and $1+g^{i_te+a}=g^{j_te}$ for $t=1,2,3$. First, note that $j_t\neq 0$ for all $t$ because $0\not\in C_a$. Moreover, we obtain the following two equations which result from $1+g^{i_te+a}=g^{j_te}$ for $t=1,2,3$
\begin{equation} \label{first equation case a0}
 1-g^{j_1e}-g^{(i_1-i_2)e}+g^{(i_1-i_2+j_2)e}=0, \ \text{and}
\end{equation}
\begin{equation} \label{second equation case a0}
1-g^{j_1e}-g^{(i_1-i_3)e}+g^{(i_1-i_3+j_3)e}=0.
\end{equation}
Suppose that there are four distinct terms in one of the equations above, assume that is (\ref{first equation case a0}). By (\ref{necessary for (a,0)}) and Theorem \ref{equivalence of equations over two fields}, we have
$$1-\zeta_k^{j_1}-\zeta_k^{i_1-i_2}+\zeta_k^{i_1-i_2+j_2}=0.$$
By Result \ref{vanishing sum of small length}, the left-hand-side sum cancels in pairs, which is impossible because $k$ is odd and all terms in the sum are distinct. Thus, we cannot have all four terms different in both (\ref{first equation case a0}) and (\ref{second equation case a0}). In (\ref{first equation case a0}), we have either $i_1-i_2=j_1$ or $i_1-i_2+j_2=0$. In (\ref{second equation case a0}), we have either $i_1-i_3=j_1$ or $i_1-i_3+j_3=0$. Due to the difference between three pairs $(i_t,j_t)$, $t=1,2,3$, we can only have two cases: $i_1-i_2=j_1$ and $i_1-i_3+j_3=0$, or $i_1-i_2+j_2=0$ and $i_1-i_3=0$. The below argument works the same for both cases. Assuming that the first case happens, we have, by (\ref{first equation case a0}) and (\ref{second equation case a0}),
$$1-2g^{j_1e}+g^{(j_1+j_2)e}=0 \ \text{and} \ 2-g^{j_1e}-g^{-j_3e}=0,$$
Equivalently
\begin{equation} \label{a system of equations}
 2-g^{-j_1e}-g^{j_2e}=0 \ \text{and} \ 2-g^{j_1e}-g^{-j_3e}=0.
\end{equation}
Hence $g^{-j_1e}+g^{j_2e}-g^{j_1e}-g^{-j_3e}=0$, which implies 
\begin{equation} \label{key equation for the case k is odd}
1+g^{(j_1+j_2)e}-g^{2j_1e}-g^{(j_1-j_3)e}=0.
\end{equation}
We claim that the numbers $0, j_1+j_2, 2j_1, j_1-j_3$ are pairwise different. As $j_1,j_2,j_3$ are pairwise different, the claim is equivalent to $2j_1 \neq 0$, $j_1+j_2 \neq 0, j_2+j_3 \neq 0$ and $j_1+j_3 \neq 0$. Firstly, $k$ odd and $j_1\neq 0$ implies $2j_1 \neq 0$. Secondly, if $j_1+j_2 = 0$, then the first equation in (\ref{a system of equations}) implies $2-2g^{-j_1e}=0$, so $j_1=0$ (note that $p>2$ by (\ref{necessary for (a,0)})), impossible. Thirdly, if $j_2+j_3=0$, then (\ref{key equation for the case k is odd}) implies $1-g^{2j_1e}=0$, so $j_1=0$. Lastly, if $j_1+j_3=0$, then the second equation in (\ref{a system of equations}) implies $2-2g^{j_1e}=0$, so $j_1=0$, a contradiction. Now by (\ref{necessary for (a,0)}) and Theorem \ref{equivalence of equations over two fields}, the equation (\ref{key equation for the case k is odd}) implies
$$1+\zeta_k^{(j_1-j_2)e}-\zeta_k^{2j_1}-\zeta_k^{(j_1-j_3)e}=0.$$
By Result \ref{vanishing sum of small length}, the left-hand-side sum cancels in pairs, impossible as $k$ is odd and the terms $0, j_1-j_2, 2j_1, j_1-j_3$ are pairwise different.
\end{proof}

\medskip

\begin{thm} \label{result for (a,a)}
Let $a \in \{1,...,e-1\}$. If
\begin{equation} \label{necessary for (a,a)}
p>\left(\frac{4k}{\varphi(k)}\right)^\frac{\varphi(k)}{2\ord_k(p)},
\end{equation}
then
\begin{equation} \label{conclusion for (a,a)}
(a,a) \leq \begin{cases} 3 \ \text{if} \ 2 \in C_a \ \text{and} \ 2\mid k ,\\
2 \ \text{if} \ 2 \not\in C_a \ \text{and} \ 2\mid k ,\\
2 \ \text{if} \ 2\nmid k .
\end{cases}
\end{equation}
\end{thm}

\begin{proof}
For each $0\leq i,j \leq k-1$ with $1+g^{ie+a}=g^{je+a}$, we have $1+g^{-ie-a}=g^{(j-i)e}$. Thus $(a,a)=(-a,0)$ and the conclusion follows directly from Theorem \ref{result for (a,0)}.
\end{proof}

\medskip

\begin{thm} \label{result of (a,b)}
Let $a\neq b \in \{1,...,e-1\}$. If
\begin{equation} \label{necessary for (a,b)}
p>\left( \frac{14k}{\varphi(k)}\right)^\frac{\varphi(k)}{2\ord_k(p)},
\end{equation}
then
$$(a,b) \leq 2.$$
\end{thm}

\medskip

This theorem is proved by contradiction. Let $(i_1,j_1), (i_2,j_2), (i_3,j_3)$ be three different pairs with $0 \leq i_t, j_t \leq k-1$ and $1+g^{i_te+a}=g^{j_te+b}$ for $t=1,2,3$. The following lemma states a simple relation between $i_t$'s and $j_t$'s which will be used repeatedly later.

\begin{lem} \label{relation between it, jt}
Let $i_t, j_t, t=1,2,3$, be defined as above, then the numbers $i_1-j_1, i_2-j_2, i_3-j_3$ are pairwise different.
\end{lem}
\begin{proof}
Suppose that $i_1-j_1=i_2-j_2$. We have $i_1-i_2=j_1-j_2$. Subtracting two equations $1+g^{i_1e+a}=g^{j_1e+b}$ and $1+g^{i_2e+a}=g^{j_2e+b}$, we obtain 
$$g^{i_2e+a}=g^{j_2e+b},$$ 
a contradiction as $C_a \cap C_b=\emptyset$.
\end{proof}

\noindent
\textit{Proof of Theorem \ref{result of (a,b)}}.
Let $(i_1,j_1), (i_2,j_2), (i_3,j_3)$ be three different pairs so that $0\leq i_t, j_t \leq k-1$ and $1+g^{i_te+a}=g^{j_te+b}$ for $t=1,2,3$. We have
$$g^a(g^{i_1e}-g^{i_2e})=g^b(g^{j_1e}-g^{j_2e}),$$ 
$$g^b(g^{j_1e}-g^{j_3e})=g^a(g^{i_1e}-g^{i_3e}).$$
Multiplying these two equations, we obtain 
\begin{equation} \label{first equation for (a,b)}
g^{(i_1+j_2)e}+g^{(i_2+j_3)e}+g^{(i_3+j_1)e}-g^{(i_1+j_3)e}-g^{(j_2+j_1)e}-g^{(i_3+j_2)e}=0.
\end{equation}
Write $f(x)=\sum_{i=0}^{k-1}a_ix^i$, where $f(g^e)$ is equal to the left-hand-side of (\ref{first equation for (a,b)}). Each $a_i$ is an integer in $[-3,3]$ and $\sum_{i=0}^{k-1}a_i=0$ and $\sum_{i=0}^{k-1}|a_i|\leq 6$. We claim that $\sum_{i=0}^{k-1} a_i^2 \leq 14$. Note that $\sum_{i=0}^{k-1}a_i^2$ is largest when one term $a_i^2$ is largest possible and other terms $a_j^2$ are smallest possible. First, there are no $i\neq j$ with $|a_i|=|a_j|=3$. Otherwise, we have $g^{(i_1+j_2)e}=g^{(i_2+j_3)e}=g^{(i_3+j_1)e}$ and $g^{(i_1+j_3)e}=g^{(i_2+j_1)e}=g^{(i_3+j_2)e}$, and (\ref{first equation for (a,b)}) implies $3(g^{(i_1+j_2)e}-g^{(i_1+j_3)e})=0$. Since $j_2 \neq j_3$, we have $p=3$, contradicting with (\ref{necessary for (a,b)}) because $p >\sqrt{14}>3$. Therefore, the sum $\sum_{i=0}^{k-1}a_i^2$ is largest when there are three nonzero terms, one equal to $(\pm 3)^2$, one equal to $(\pm 2)^2$ and one equal to $(\pm 1)^2$, that is
$$\sum_{i=0}^{k-1}a_i^2\leq 9+4+1=14.$$
Now combining (\ref{first equation for (a,b)}), (\ref{necessary for (a,b)}) and  Theorem \ref{equivalence of equations over two fields}, we obtain
\begin{equation} \label{second equation for (a,b)}
f(\zeta_k)=\zeta_k^{i_1+j_2}+\zeta_k^{i_2+j_3}+\zeta_k^{i_3+j_1}-\zeta_k^{i_1+j_3}-\zeta_k^{i_2+j_1}-\zeta_k^{i_3+j_2}=0.
\end{equation}
Note that $f(\zeta_k)$ is a vanishing sum of roots of unity of length at most $6$. By Result \ref{vanishing sum of small length}, $f(\zeta_k)$ contains a subsum similar to $1+i^2$, or $f(\zeta_k)$ contains two subsums each of which is similar to $1+\zeta_3+\zeta_3^2$, or $f(\zeta_k)$ itself is similar to either $1+\zeta_5+\zeta_5^2+\zeta_5^3+\zeta_5^4$ or $-\zeta_3-\zeta_3^2+\zeta_5+\zeta_5^2+\zeta_5^3+\zeta_5^4$.\\

\noindent
\textbf{Case 1.} $f(\zeta_k)$ contains a subsum similar to $1 + (-1)$.\\
Discarding the empty sum, the new $f(\zeta_k)$ is a vanishing sum of $4$ roots of unity. By Result \ref{vanishing sum of small length} again, $f(\zeta_k)$ cancels in pairs. Thus, the original sum $f(\zeta_k)$ cancels in pairs. Note that none of the first three terms in (\ref{second equation for (a,b)}) is canceled by any of the last three terms. Otherwise, let's say $\zeta_k^{i_1+j_2}$ is canceled by one of the last three terms. By the difference between the $i_t$'s and $j_t$'s, we can only have $i_1+j_2=i_2+j_1$, which implies $i_1-j_1=i_2-j_2$, contradicting with Lemma \ref{relation between it, jt}. Thus, the first three terms of $f(\zeta_k)$ cancel in pairs, impossible.

\medskip

\noindent
\textbf{Case 2.} $f(\zeta_k)$ is similar to $1+\zeta_5+\zeta_5^2+\zeta_5^3+\zeta_5^4$. \\
Note that by Case $1$, the sets $\{i_1+j_2, i_2+j_3, i_3+j_1\}$ and $\{i_1+j_3, i_2+j_1, i_3+j_2\}$ are disjoint. As $f(\zeta_k)$ has length $5$, we can assume that the first two terms in $f(\zeta_k)$ are the same, say
$f(\zeta_k)=2\zeta_k^{i_1+j_2}+\zeta_k^{i_3+j_1}-\zeta_k^{i_1+j_3}-\zeta_k^{i_2+j_1}-\zeta_k^{i_3+j_2}$. Hence, $f(\zeta_k)$ is similar to  the sum $2+\zeta_k^{i_3+j_1-i_1-j_2}-\zeta_k^{j_3-j_2}-\zeta_k^{i_2+j_1-i_1-j_2}-\zeta_k^{i_3-i_1}$.
It is impossible that this sum has the form $1+\zeta_5+\zeta_5^2+\zeta_5^3+\zeta_5^4$.

\medskip

\noindent
\textbf{Case 3.} $f(\zeta_k)$ contains two subsums each of which is similar to $1+\zeta_3+\zeta_3^2$. \\
Due to symmetry, we can consider two possibilities for these two subsums.

\medskip

\noindent
\underline{Subcase 1}. The subsums are $\zeta_k^{i_1+j_2}+\zeta_k^{i_2+j_3}+\zeta_k^{i_3+j_1}$ and $\zeta_k^{i_1+j_3}+\zeta_k^{i_2+j_1}+\zeta_k^{i_3+j_2}$.\\
We obtain $1+\zeta_k^{i_2+j_3-i_1-j_2}+\zeta_k^{i_3+j_1-i_1-j_2}=1+\zeta_k^{i_2+j_1-i_1-j_3}+\zeta_k^{i_3+j_2-i_1-j_3}$ and both sums have the form $1+\zeta_3+\zeta_3^2$.
Thus $3 \mid k$,
\begin{equation} \label{relation between exponents 1}
\{i_2+j_3-i_1-j_2, i_3+j_1-i_1-j_2\}= \{k/3, 2k/3\}, \ \text{and} \\
\end{equation}
\begin{equation} \label{relation between exponents 2}
\{i_2+j_1-i_1-j_3, i_3+j_2-i_1-j_3\}=\{k/3, 2k/3\}.
\end{equation}
Since $k/3+2k/3=0$, we have $(i_2+j_3-i_1-j_2)+(i_3+j_1-i_1-j_2)=0$ and $(i_2+j_1-i_1-j_3)+(i_3+j_2-i_1-j_3)=0$, which implies
\begin{equation}\label{extra relation 1}
2(i_1+j_2)=(i_2+j_3)+(i_3+j_1)
\end{equation}
and
\begin{equation} \label{extra relation 2}
2(i_1+j_3)=(i_2+j_1)+(i_3+j_2).
\end{equation}
Subtracting (\ref{extra relation 1}) and (\ref{extra relation 2}), we obtain $j_2-j_3=k/3$. Now, the equation (\ref{relation between exponents 1}) gives $i_2-i_1=2k/3$ and the equation (\ref{relation between exponents 2}) gives $i_3-i_1=k/3$. We obtain $i_2-i_3=j_2-j_3=k/3$, so $i_2-j_2=i_3-j_3$,
contradicting with Lemma \ref{relation between it, jt}.

\medskip

\noindent
 \underline{Subcase 2}. The subsums are $\zeta_k^{i_1+j_2}+\zeta_k^{i_2+j_3}-\zeta_k^{i_2+j_1}$ and $\zeta_k^{i_3+j_1}-\zeta_k^{i_1+j_3}-\zeta_k^{i_3+j_2}$.\\
We obtain $1+\zeta_k^{i_2+j_3-i_1-j_2}-\zeta_k^{i_2+j_1-i_1-j_2}=1+\zeta_k^{i_3+j_2-i_1-j_3}-\zeta_k^{i_3+j_1-i_1-j_3}$ and both sums are equal to $1+\zeta_3+\zeta_3^2$. Thus $6 \mid k$ and the two sums $1+\zeta_k^{i_2+j_3-i_1-j_2}-\zeta_k^{i_2+j_1-i_1-j_2}$ and $1+\zeta_k^{i_3+j_2-i_1-j_3}-\zeta_k^{i_3+j_1-i_1-j_3}$ have form $1+\zeta_k^{k/3}-\zeta_k^{k/6}$ or $1+\zeta_k^{2k/3}-\zeta_k^{5k/6}$. If these two sums have the same form, then $i_2+j_1-i_1-j_2=i_3+j_1-i_1-j_3$, so $i_2-j_2=i_3-j_3$, contradicting with Lemma \ref{relation between it, jt}. Thus, the two sums have different forms. Noting that $k/6+2k/3=5k/6$ and $5k/6+k/3=k/6$, we have
$$(i_2+j_1-i_1-j_2)+(i_3+j_2-i_1-j_3)=(i_3+j_1-i_1-j_3),$$
so $i_2=i_1$, a contradiction.

\medskip

\noindent
\textbf{Case 4.} $f(\zeta_k)$ is similar to $-\zeta_3-\zeta_3^2+\zeta_5+\zeta_5^2+\zeta_5^3+\zeta_5^4$.\\
A reduced sum of this sum is
 $$S'=1+\zeta_5+\zeta_5^2+\zeta_5^3-\zeta_3\zeta_5^{-1}-\zeta_3^2\zeta_5^{-1}.$$
Let $S$ be the reduced sum obtained from $f(\zeta_k)$ as follows
$$S=1+\zeta_k^{i_2+j_3-i_1-j_2}+\zeta_k^{i_3+j_1-i_1-j_2}-\zeta_k^{j_3-j_2}-\zeta_k^{i_2+j_1-i_1-j_2}-\zeta_k^{i_3-i_1}.$$
Dividing by a common divisor if necessary, we can assume that the greatest common divisor between $k$ and all the exponents of $\zeta_k$ occurring in $S$ is $1$. This implies $e(S)=k$. In view of Result \ref{squarefree}, we can assume that $k$ is square-free.
 Since $S$ and $S'$ are similar reduced sums, we have $ S=S' \zeta_{30}^t$ with $t \in \{0,1,11,12,18,24\}$ (the possible values of $t$ are obtained from the fact that $1$ appears in $S$). The $6$ possibilities are
\begin{itemize}
\item[(i)] $S'=1+\zeta_5+\zeta_5^2+\zeta_5^3-\zeta_{15}^2-\zeta_{15}^7.$
\item[(ii)] $S'\zeta_{30}=1+\zeta_3^2-\zeta_{15}^8-\zeta_{15}^{11}-\zeta_{15}^{14}-\zeta_{15}^2.$
\item[(iii)] $S'\zeta_{30}^{11}=1+\zeta_3-\zeta_{15}^{13}-\zeta_{15}-\zeta_{15}^4-\zeta_{15}^7$
\item[(iv)]
$S'\zeta_{30}^{12}=1+\zeta_5^2+\zeta_5^3+\zeta_5^4-\zeta_{15}^8-\zeta_{15}^{13}.$
\item[(v)]
$S'\zeta_{30}^{18}=1+\zeta_5+\zeta_5^3+\zeta_5^4-\zeta_{15}^{11}-\zeta_{15}.$
\item[(vi)]
$S'\zeta_{30}^{24}=1+\zeta_5+\zeta_5^2+\zeta_5^4-\zeta_{15}^{14}-\zeta_{15}^4.$
\end{itemize}

Suppose that $k$ is odd. We obtain $k=15$ and the sum $S$ has the exact form as one of the $6$ possibilities above, impossible as the sum of the coefficients in any of these possibilities is nonzero. Therefore, $k$ is even. Note that $e(S'\zeta_{30}^t)=15$ in any case and we can write $\zeta_{30}=-\zeta_{15}^8$. So $k=30$. Multiplying all the terms in both sides of the equation $S=S'\zeta_{30}^t$, we obtain
$$\zeta_{30}^{2(i_2+i_3+j_1+j_3)-4(i_1+j_2)+15}=\zeta_{30}^{24+6t},$$ 
which implies $2(i_2+j_3+j_1+j_3)-4(i_1+j_2)-6t \equiv 9 \pmod{30}$, impossible. 
This completes the proof of Theorem \ref{result of (a,b)}.  $\hfill\Box$

\medskip

\begin{rmk}
Summarizing the results of Theorem \ref{the number (0,0)}, Theorem \ref{result for (0,a)}, Theorem \ref{result for (a,0)}, Theorem \ref{result for (a,a)} and Theorem \ref{result of (a,b)}, we obtain $(a,b) \leq 3$ if $p> \left(14k/\varphi(k) \right)^{\varphi(k)/(2\ord_k(p))}$. The inequality $p>\left(\sqrt{14}\right)^{k/\ord_k(p)}$ is sufficient for $p> \left(14k/\varphi(k) \right)^{\varphi(k)/(2\ord_k(p))}$, due to (\ref{necessary resulting}). Thus, Main Theorem 1 is proved.
\end{rmk}

\bigskip

\section{The Case {\boldmath $k$} is Prime}

In this section, we always assume that $k$ is a prime and $a\neq b \in \{1,\dots,e-1\}$. 
We recall the our main result on this case.

\begin{main2} 
Let $q$ be a power of a prime $p$. Let $e$ and $k$ be nontrivial divisors of $q-1$ such that $q=ek+1$ and $k$ is a prime. If 
\begin{equation} \label{nec k prime}
p>(3^{k-1}k)^{1/\ord_k(p)},
\end{equation}
then 
\begin{equation} \label{conclusion (a,b) k prime}
(a,b) \leq 2 \ \ \text{for all} \ \ a,b\in \mathbb{Z}.
\end{equation}
\end{main2}

\medskip

Similar to the proof of Main Theorem 1, the proof of Main Theorem 2 
is divided into the cases $(0,0)$, $(0,a)$, $(a,0)$, $(a,a)$ and $(a,b)$
and Main Theorem 2 is just a simplified consequence of the results for the different cases.
We remark that only in the cases $(0,a)$ and $(a,b)$, 
we obtain better upper bounds for these numbers than the bounds obtained in the last section. 
We restate the results for $(0,0)$, $(a,0)$ and $(a,a)$ here for the completeness of the proof.

\begin{cor} \label{result (0,0) k prime}
If 
$$p>\left( \frac{3k}{k-1}\right)^{\frac{k-1}{2\ord_k(p)}},$$
then
$$(0,0)=\begin{cases} 0 \ \ \text{if} \ \ 2\not\in C_0, \\ 1 \ \ \text{if} \ \ 2 \in C_0.\end{cases}$$
\end{cor}

\begin{proof}
This theorem is a direct consequence of Theorem \ref{the number (0,0)}. 
Note that the case $6 \mid k$ cannot occur because $k$ is a prime.
\end{proof}

\medskip

\begin{cor} \label{result for (a,0) k is prime}
If
$$p>\left( \frac{4k}{k-1}\right)^{\frac{k-1}{2\ord_{k}(p)}},$$
then
$$(a,0) \leq 2 \ \ \text{and} \ \ (a,a) \leq 2.$$
\end{cor}

\begin{proof}
If $k$ is even, then $k=2$ and it is trivial that $(a,0) \leq 2$. 
If $k$ is odd, then $(a,0) \leq 2$ by Theorem \ref{result for (a,0)} (the case $k$ is odd).
Lastly, note that $(a,a)=(-a,0) \leq 2$.
\end{proof}

\medskip

\begin{thm} \label{result for (0,a) k is prime}
If
\begin{equation} \label{necessary condition for (0,a) k is prime}
p>\left( 2^{k-1}k \right)^{{1}/{\ord_{k}(p)}} ,
\end{equation}
then 
\begin{equation} \label{conclusion for (0,a) k is prime}
(0,a) \leq 2.
\end{equation}
\end{thm}

\begin{proof}
Each equation $1+g^{ie}=g^{je+a}$ induces another equation $1+g^{-ie}=g^{(j-i)e+a}$, 
and these equations are different only if $i \neq 0$. Moreover, if $i=0$, then $2=g^{je+a} \in C_a$.

First, suppose that $2 \in C_a$ and $(0,a) \geq 3$. We have $g^{le+a}=2$ for some $0 \leq l \leq k-1$. There exist $0\leq i,j \leq k-1$ with $i\neq 0$ and $j\neq l$ such that $1+g^{ie}=g^{je+a}$. Writing $t=j-l$, we obtain
 $$1+g^{ie}-2g^{te}=0.$$
Note that the numbers $0, i, t$ are pairwise different. 
Write $1+x^i-2x^t$ in the polynomial form $f(x)=\sum_{i=0}^{k-1}a_ix^i$.
Note that, using the notation of Theorem \ref{equivalence of equations over two fields 2} (a), we have $A=2$.
Thus, by Theorem \ref{equivalence of equations over two fields 2} (a) and (\ref{necessary condition for (0,a) k is prime}),
we have $f(\zeta_k)=0$, as $f(g^e)=0$. Hence
$$f(\zeta_k)=1+\zeta_k^i-2\zeta_k^t=0.$$
By Result \ref{vanishing sum of small length}, this happens only when 
the terms in $f(\zeta_k)$ cancel in pairs or $f(\zeta_k)$ is similar to $1+\zeta_3+\zeta_3^2$, both of which are not possible.

\medskip

Next, suppose that $2 \not\in C_a$ and $(0,a) \geq 3$. Similar to the proof of Theorem \ref{result for (0,a)}, we obtain the  equation
$$1+g^{i_1e}-g^{(j_1-j_2)e}-g^{(j_1-j_2+i_2)e}=0,$$
where the two pairs $(i_1,j_1)$ and $(i_2,j_2)$ are different and satisfy
$$i_1 \neq 0,\ i_2 \neq 0, \ (i_2,j_2) \neq (-i_1, j_1-i_1)\} \ \text{and} \ (i_1,j_1)\neq (-i_2,j_2-i_2).$$ 
Write $f(x)=1+x^{i_1}-x^{j_1-j_2}-x^{j_1-j_2+i_2}$. 
Note that $f(x)$ is a polynomial with exactly $4$ nonzero coefficients, 
as the numbers $0, i_1, j_1-j_2$ and $j_1-j_2+i_2$ are pairwise different 
(follows from the proof of Theorem \ref{result for (0,a)}). 
Thus, by Theorem \ref{equivalence of equations over two fields 2} (a) and (\ref{necessary condition for (0,a) k is prime}),
we have
$$f(\zeta_k)=1+\zeta_k^{i_1}-\zeta_k^{j_1-j_2}-\zeta_k^{j_1-j_2+i_2}=0.$$
 By Result \ref{vanishing sum of small length}, the terms in $f(\zeta_k)$ cancel in pairs. 
 This implies in $2 \mid k$ and $i_1=i_2=k/2$, or $j_1=j_2$ and $i_1=i_2$, both of which are not possible.
\end{proof}

\begin{rmk} \label{comparison with general result on (0,a)}
The bound  (\ref{necessary condition for (0,a) k is prime}) is not better than 
the previous bound in (\ref{necessary condition for (0,a)}) (in fact, they are very close). 
However, the conclusion (\ref{conclusion for (0,a) k is prime}) is better than the conclusion (\ref{conclusion for (0,a)}).
\end{rmk}

\medskip

\begin{thm} \label{result for (a,b) k is prime}
If
\begin{equation} \label{necessary for (a,b) in the case k is prime}
p>(3^{k-1}k)^{1/\ord_k(p)},
\end{equation}
then 
$$(a,b) \leq 2.$$
\end{thm}

\begin{proof}
Similar to the proof of Theorem \ref{result of (a,b)}, we obtain the equation
\begin{equation} \label{first equation for (a,b) k prime}
g^{(i_1+j_2)e}+g^{(i_2+j_3)e}+g^{(i_3+j_1)e}-g^{(i_1+j_3)e}-g^{(j_2+j_1)e}-g^{(i_3+j_2)e}=0,
\end{equation}
where $(i_t, j_t)$, $t=1,2,3$, are pairwise different pairs each of which satisfy $1+g^{i_te+a}=g^{j_te+b}$.
Write the left-hand-side of (\ref{first equation for (a,b) k prime}) as $\sum_{i=0}^{k-1}a_ig^{ie}$ and set $f(x)=\sum_{i=0}^{k-1}a_ix^i$. 
We have $\sum_{i=0}^{k-1}a_i=0$
and, using the notation of Theorem \ref{equivalence of equations over two fields 2} (a), we have $A\ge 3$.
Hence Theorem \ref{equivalence of equations over two fields 2} (a) and (\ref{necessary for (a,b) in the case k is prime}) imply
$$
f(\zeta_k)=\zeta_k^{i_1+j_2}+\zeta_k^{i_2+j_3}+\zeta_k^{i_3+j_1}-\zeta_k^{i_1+j_3}-\zeta_k^{i_2+j_1}-\zeta_k^{i_3+j_2}=0.
$$
This is impossible by the proof of Theorem \ref{result of (a,b)}.
\end{proof}

\medskip

\begin{rmk} \label{comparison with general result (a,b)}
Therem \ref{result for (a,b) k is prime} is an improved version of Theorem \ref{result of (a,b)}, as the bound  (\ref{necessary for (a,b) in the case k is prime}) is better than the one in (\ref{necessary for (a,b)}). Furthermore, Theorem \ref{result (0,0) k prime}, Theorem \ref{result for (0,a) k is prime}, Theorem \ref{result for (a,0) k is prime} and Theorem \ref{result for (a,b) k is prime} prove Main Theorem 2.
\end{rmk}

\bigskip

\end{document}